\documentclass{scrartcl}
\usepackage{mathtools}
\usepackage{amsthm}
\usepackage{tikz-cd}
\usepackage{dsfont}
\usepackage[
  style=alphabetic,
  maxalphanames=5,
  maxbibnames=1000,
  backend=biber
  ]{biblatex}
\usepackage{comment}
\usepackage{amssymb}
\usepackage{csquotes}
\usepackage{hyperref}
\usepackage{cleveref}
\usepackage{todonotes}
\usepackage[new]{old-arrows}

\bibliography{references.bib}

\title{On orthogonal factorization systems and double categories}
\author{Branko Juran \thanks{Department of Mathematical Sciences, University of Copenhagen, Denmark \\ bj@math.ku.dk}}
\date{\vspace{-5ex}}

\newtheorem{lemma}{Lemma}[section]
\newtheorem{theorem}[lemma]{Theorem}
\newtheorem{thmx}{Theorem}

\newtheorem{proposition}[lemma]{Proposition}
\newtheorem{corollary}[lemma]{Corollary}
\theoremstyle{definition}
\newtheorem{example}[lemma]{Example}
\newtheorem{definition}[lemma]{Definition}
\newtheorem{construction}[lemma]{Construction}
\newtheorem{remark}[lemma]{Remark}
\newtheorem{notation}[lemma]{Notation}

\DeclareMathOperator{\const}{const}

\DeclareMathOperator{\op}{op}
\DeclareMathOperator{\precorners}{cnr}
\DeclareMathOperator{\corners}{Cnr}
\DeclareMathOperator{\map}{map}
\newcommand{\Sq}{\operatorname{Sq}^{\mathrm{oplax}}}
\DeclareMathOperator{\Cat}{Cat}
\DeclareMathOperator{\id}{id}
\DeclareMathOperator{\OFS}{OFS} 
\DeclareMathOperator{\Fun}{Fun}
\DeclareMathOperator{\PSh}{PSh}

\DeclareMathOperator*{\colim}{colim}
\DeclareMathOperator{\Aut}{Aut}

\DeclareMathOperator{\Ortho}{Ortho}

\DeclareMathOperator{\opGray}{opGray}
\DeclareMathOperator{\Span}{Span}

\newcommand{\cartright}{(cart,right)-fibration}
\newcommand{\cocartright}{(cocart,right)-fibration}

\newcommand{\leftcart}{(left,cart)-fibration}
\newcommand{\sOFSS}{\catS^{\dag}}
\newcommand{\sOFS}{{\cat}^{\dag}}
\newcommand{\join}[2]{#1 \star #2}
\newcommand{\overslice}[2]{{#1}_{/{#2}}}
\newcommand{\underslice}[2]{{#1}_{#2 /}}
\newcommand{\An}{\mathcal S}
\newcommand{\ArDC}[1]{\mathbb{A}\mathrm{r}\left(#1\right)}
\newcommand{\Ar}[1]{\operatorname{Ar}\left( #1 \right)}
\newcommand{\LinOrd}[1]{\left[ #1 \right]}
\DeclareMathOperator{\DCLR}{Fact}
\newcommand{\DblCatOF}{\DblCat^{\operatorname{OF}}}
\newcommand{\catL}{\mathcal C_{\mathrm{eg}}}
\newcommand{\catR}{\mathcal C_{\mathrm{in}}}
\newcommand{\catSL}{\mathcal D_{\mathrm{eg}}}
\newcommand{\catSR}{\mathcal D_{\mathrm{in}}}
\newcommand{\cat}{\mathcal C}
\newcommand{\catS}{\mathcal D}
\newcommand{\CatInftyOplax}{\CatInfty^{(2)}}
\newcommand{\doublecat}{\mathbb C}
\newcommand{\doublecatS}{\mathbb D}
\newcommand{\boxproduct}{\boxtimes}
\newcommand{\Catn}[1]{\Cat_{#1}}
\newcommand{\CatInfty}{\Catn{1}}
\newcommand{\core}[1]{#1^{\simeq}}

\newcommand{\CartRight}[1]{\mathrm{CaR}\left( #1 \right)}
\newcommand{\CocartRight}[1]{\mathrm{CoR}\left( #1 \right)}

\newcommand{\LeftCart}[1]{\mathrm{LCa}\left( #1 \right)}

\newcommand{\spine}[1]{I_{#1}}
\newcommand{\genSpine}[2][1]{I_{#1}\left( #2 \right)}
\newcommand{\OFSAd}{\OFS^{\bot}}

\newcommand{\DblCat}{\operatorname{DCat}}

\newcommand{\Gray}{\otimes}
\newcommand{\horop}[1]{{#1}^{1 \op}}
\newcommand{\verop}[1]{{#1}^{2 \op}}
\newcommand{\fullop}[1]{{#1}^{12 \op,\mathrm{swap}}}
\newcommand{\swapop}[1]{{#1}^{\mathrm{swap}}}
\newcommand{\Segalmap}{\rho}
\newcommand{\DblCatAd}{\DblCat^{\bot}}

\def\prodOFS{\bar{\times}}

\begin{document}
\maketitle

\begin{abstract}
    We prove that the $\infty$-category of orthogonal factorization systems embeds fully faithfully into the $\infty$-category of double $\infty$-categories. Moreover, we prove an (un)straightening equivalence for double $\infty$-categories, which restricts to an (un)straightening equivalence for op-Gray fibrations and curved orthofibrations of orthogonal factorization systems. 
\end{abstract}

\tableofcontents

\section{Introduction}
Orthogonal factorization systems and double categories are very classical objects in category theory, their study goes back to work of MacLane \cite{sm50} and Ehresmann \cite{ehr63}, respectively. Their $\infty$-categorical analogs, introduced by Joyal \cite{joy08} and by Haugseng \cite{hau13}, play an equally important role in higher category theory. 
Both concepts deal with categorical structures equipped with two distingushed classes of morphisms.
An orthogonal factorization system is a category together with the choice of two classes of morphisms, so that any morphism can uniquely be factored as a composite of one morphism from the first class, followed by one from the second.
Similarly, we can think of a double category as a category with two different types of morphisms, \emph{vertical} and \emph{horizontal} morphisms. These two different types of morphisms cannot be composed with each other but part of the data are so-called \emph{squares} which witness compatibilities between them.
From this point of view, it seems like double categories are a generalization of orthogonal factorization systems. 

The goal of this paper is to make this precise in the context of $\infty$-categories.
We will construct a functor \[ \DCLR \colon \OFS \longhookrightarrow \DblCat \] from the $\infty$-category of orthogonal factorization systems into the $\infty$-category of double $\infty$-categories (\Cref{def-DCLR}).
It sends an orthogonal factorization system to the double $\infty$-category which has the objects of the underlying $\infty$-category as its objects, horizontal morphisms are morphisms belonging to the first class, vertical morphisms are morphisms from the second class and squares are commutative squares. The essential image of this functor consists of the double $\infty$-categories where there is a unique square filling every choice of a \enquote{bottom left corner}: 
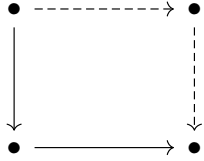
\[ \begin{tikzcd}[sep=huge]
    \bullet \ar[r,dashed] \ar[d] & \bullet \ar[d,dashed] \\
    \bullet \ar[r] & \bullet
\end{tikzcd}  \] 
This precisely encodes that the \enquote{wrong order} composition of morphisms from each class can uniquely be rewritten as composition in the \enquote{correct order}, arguably the most important feature of an orthogonal factorization system. We will then prove the following: 
\begin{thmx}[{\Cref{main-theorem}}] \label{intro-thm}
    The functor
    \[ \DCLR \colon \OFS \longhookrightarrow \DblCat \] from the $\infty$-category of orthogonal factorization systems into the $\infty$-category of double $\infty$-categories is fully faithful.
    The essential image consists of the double $\infty$-categories fulfilling the equivalent conditions from \Cref{equ-def-DblOF}.
\end{thmx}

A similar result in the context of ordinary $1$-categories has recently been obtained by \v{S}t\v{e}pán \cite[Thm. 3.30]{ms23}.

We use \Cref{intro-thm} to deduce several other results about orthogonal factorization systems.
Firstly, we verify that for $\sOFS$ an orthogonal factorization system, the functor $\DCLR$ from \Cref{intro-thm} induces an equivalence between  curved orthofibrations (or op-Gray fibrations, respectively) over $\sOFS$, defined in \cite{hhln23a}, and \cocartright s (or \cartright s, respectively) over $\DCLR(\sOFS)$, defined in \cite{nui21}, (\Cref{comparefibrations}): 
\begin{align} \label{intro-Ortho=Cor} \Ortho\left(\sOFS\right)  \cong \CocartRight{\DCLR\left(\sOFS\right)} \,, \end{align}
showing that we can regard the former as a special case of the latter.
Then we prove the following (un)straightening equivalence for these fibrations:
\begin{thmx}[{\Cref{maintheorem}}] \label{intro-str-unst}
    For $\doublecat$ a double $\infty$-category, let $\core{\CocartRight{\doublecat}}$ denote the space of \cocartright s. 
    There is a natural equivalence \[ \core{\CocartRight{\doublecat}} \cong \map_{\DblCat}\left(\verop{(\doublecat)},\Sq\left(\CatInftyOplax\right)\right)  \] of functors from double $\infty$-categories to spaces. Here, $\Sq\left(\CatInftyOplax\right)$ is the large double $\infty$-category having $\infty$-categories as its objects, functors of $\infty$-categories as horizontal and vertical morphisms and natural transformations as squares.
\end{thmx}
This confirms an expectation outlined in \cite[Remark 2.14]{nui21}. Combining \eqref{intro-Ortho=Cor} and \Cref{intro-str-unst} yields an (un)straightening equivalence for fibrations of orthogonal factorization systems.

Among orthogonal factorization there are those which are called \emph{adequate} (in the sense of Barwick \cite{bar17}), those admitting certain pullbacks that make it possible to define the span category.
We check that under the embedding from \Cref{intro-thm}, an orthogonal factorization system $\sOFS$ is adequate if and only if $\horop{\DCLR(\sOFS)}$, the opposite in the horizontal direction of the associated double $\infty$-category, is also contained in the essential image of $\DCLR$. In this case, there is an equivalence \[\horop{\DCLR\left(\sOFS\right)} \cong \DCLR\left(\mathrm{Span}\left(\sOFS\right)\right) \,.\] We hence recover the span category construction of adequate orthogonal factorization systems, which usually involves a certain amount of simplicial combinatorics. In fact we do so by computing the whole automorphism group of the $\infty$-category of adequate orthogonal factorization systems $\OFSAd$, refining \cite[Cor. 5.7]{hhln23a}:
\begin{thmx}[{\Cref{Aut(AOFS)}}] \label{intro-Aut}
    There is an equivalence of groups
    \[
        \Aut\left(\OFSAd\right) \cong \mathbb{Z}/2\mathbb{Z}
    \]
    with generator given by the span category functor.
\end{thmx}
\subsection*{Organization of the paper}
In \Cref{section-preliminaries}, we collect some background material on double $\infty$-categories and orthogonal factorization systems. \Cref{section-mainthm} is dedicated to the proof of \Cref{intro-thm}. In \Cref{section-fibrations}, we study fibrations, proving \Cref{intro-str-unst}. Finally, we discuss adequate orthogonal factorization systems in \Cref{section-adequate} and prove \Cref{intro-Aut}.
\subsection*{Acknowledgments}
I would like to thank Sil Linskens for telling me about this potential connection between double categories and orthogonal factorization systems. I am grateful to Jan Steinebrunner and Jaco Ruit for helpful conversations. I would also like to thank my supervisor Jesper Grodal for his guidance and advice. The author was supported by the DNRF through the Copenhagen Center for Geometry and Topology (DNRF151).

\subsection*{Conventions}
\begin{itemize}
    \item We use the theory of $\infty$-categories as developed by Lurie in \cite{lur09} and \cite{lur17}. However, we make no essential use of the concrete model and the arguments should apply in most other models. Since most of the categories appearing in this paper are higher categories, we use the term \emph{category} to refer to an $\infty$-category and say $(1,1)$-category if the mapping spaces are discrete.
    \item We fix three nested Grothendieck universes $U \in V \in W$. We call a category \emph{very large} if it is defined in $W$ and \emph{large} if is defined in $V$. If we just say category, we assume that it is defined in $U$.
    \item We write $\An$ for the large category of spaces (or animae, homotopy types, $\infty$-groupoids, ...) and $\CatInfty$ for the large category of categories. 
    Moreover, we write $(-)^\simeq \colon \CatInfty \to \An$ for the core functor. 
    \item The simplex category is denoted by $\Delta$. Its objects are finite posets $\LinOrd{n} = \left\{ 0,1, \dots , n \right\}$.     
    We write $d_i \colon \LinOrd{n-1} \to \LinOrd{n}$ for the injective map omitting $i$ and $\Segalmap_i \colon \LinOrd{1} \to \LinOrd{n}$ for the map sending $0$ to $i-1$ and $1$ to $i$.
    
    We use the same notation to denote the category obtained from this poset via the inclusions $\operatorname{Poset} \hookrightarrow \CatInfty$.
    \item We use $\Lambda_2^2 = (0 \to 2 \leftarrow 1)$ for the cospan $(1,1)$-category (i.e. the $2$-horn of the $2$-simplex).
    \item For $\cat$ a category, we write $\Ar{\cat} =\Fun(\LinOrd{1},\cat)$ for the \emph{arrow category} of $\cat$.
    \item We write $\PSh(\mathcal C) = \Fun(\mathcal C^{\op},\An)$ for the large presheaf category of a category $\mathcal C$.
    \item A full subcategory is called \emph{reflective} if the inclusion functor admits a left adjoint. By \cite[Prop. 5.5.4.15]{lur09}, a subcategory of a large presentable category is reflective if it can be characterized as the class of local objects with respect to a small set of morphisms.
\end{itemize}

\section{Preliminaries} \label{section-preliminaries}
\begin{definition} \label{def-simplicial-space}
    A \emph{simplicial space} is a presheaf on $\Delta$, i.e. a functor \[ X \colon \Delta^{\op} \longrightarrow \An \,. \] We call it a \emph{Segal space} if \[ \prod_{i=1}^{n} \Segalmap_i \colon X_n \longrightarrow  X_1 \times_{X_0} \dots \times_{X_0} X_1 \] is an equivalence. 

    We call a Segal space \emph{complete} if \[ \begin{tikzcd}
        X_0 \ar[r,"\Delta"] \ar[d] & X_0 \times X_0 \ar[d] \\
        X_3 \ar[r] & X_1 \times X_1
    \end{tikzcd} \] is a pullback. 
    Here, the vertical morphisms are induced by the unique degeneracy map to $\LinOrd{0}$ and the bottom map is the one induced by the inclusions $\LinOrd{1} \to \LinOrd{3}$ onto $\left\{ 0,2 \right\}$ and $\left\{ 1,3 \right\}$, respectively. 
\end{definition}
\begin{remark} \label{checkcompleteness}
    Let $X$ be a Segal space and let $x,y \in X_0$. Let us write 
    \[ \map_X(x,y) = \left\{ x \right\} \times_{X_0,d_1} X_1 \times_{d_0,X_0} \left\{ y \right\} \,. \]
    Using the higher Segal conditions, one can define composition maps. It is proved in \cite{rez01} that a Segal space is complete if and only if the degeneracy map \[ X_0 \longrightarrow X_1 \] is fully faithful with essential image being the equivalences, i.e. the maps $f \colon x \to y$ such that composition with $f$ induces equivalences 
    \begin{align*} 
        f_\ast \colon \map_X(z,x) &\longrightarrow \map(z,y) \intertext{and} f^\ast \colon \map_X(y,z) &\longrightarrow \map_X(x,z)
    \end{align*}
    for every object $z$ of $X_0$.
\end{remark}
\begin{theorem}[Joyal-Tierney {\cite{jt07}}] \label{rezk-thm}
    The restricted Yoneda embedding \begin{align*}
        \CatInfty & \longrightarrow \Fun(\Delta^{\op},\An) \\
        \cat & \longmapsto \map(\LinOrd{n},\cat)
    \end{align*}
    is fully faithful with essential image the complete Segal spaces. 
\end{theorem}
The idea of the above theorem goes back to Rezk \cite{rez01}. A simple proof of the above theorem can be found in \cite{hs23}.

Before turning towards the definition of a double category, we recall the following:
\begin{proposition}[{\cite[Prop. 1.7]{bri18}}]  \label{rightfibrations}
    Let $F \colon \catS \to \cat$ be a map of Segal spaces. The following are equivalent:
    \begin{itemize}
        \item The square \[ \begin{tikzcd}[sep=huge]
            \catS_n \ar[r,"{F_n}"] \ar[d,"{d_0}"] & \cat_n \ar[d,"d_0"] \\
            \catS_0 \ar[r,"F_0"] & \cat_0
        \end{tikzcd} \]
        is cartesian.
        \item The above square is cartesian for $n=1$.
    \end{itemize}
\end{proposition}
In this case, the map is called a \emph{right fibration}. The dual notion (replacing $d_0$ by $d_n$) is called \emph{left fibration}.
\begin{definition}
    A \emph{double category} is a bisimplicial space 
    \begin{align*} 
        \doublecat \colon \Delta^{\op} \times \Delta^{\op} & \longrightarrow \An \\
        (\LinOrd{m},\LinOrd{n}) & \longmapsto \doublecat(m,n)
    \end{align*}
    such that for all $n \ge 0 $, $\doublecat(n,-)$ and $\doublecat(-,n)$ are complete Segal spaces. 

    We write $\DblCat \subset \PSh(\Delta \times \Delta)$ for the large full subcategory of double categories. 
\end{definition}
\begin{remark}
    This definition is not used consistently in the literature.
    We follow the terminology used in \cite{nui21}.
    Some authors only require $\doublecat(-,n)$ to be a (not necessarily complete) Segal space. 
    
    In particular, our definition is not the direct higher categorical analog of the classical notion of a double category. In our definition, the space of objects in the vertical and horizontal category must agree.
\end{remark}

\begin{notation}
    We write \[ \horop{(-)} , \verop{(-)} \colon \DblCat \longrightarrow \DblCat \] for the functors precomposing with the opposite functor $(-)^{\op} \colon \Delta \to \Delta$ in the first or second coordinate, respectively, and \[ \swapop{(-)} \colon \DblCat \longrightarrow \DblCat \] for the functor which exchanges the two coordinates. 
\end{notation}

\begin{construction} 
    Given two categories $\cat$ and $\catS$, we obtain a double category $\cat \boxproduct \catS$ as follows:
    \[ (\cat \boxproduct \catS)(k,l) = \map(\LinOrd{k},\cat) \times \map(\LinOrd{l},\catS)  \] 
\end{construction}

\begin{definition}
    A double category if called a \emph{$2$-category} if $\doublecat(0,-)$ is a constant simplicial space. 
    We write $\Catn{2} \subset \DblCat$ for the large full subcategory of double categories. 
\end{definition}
\begin{remark} \label{embedding-strict-2-categories}
We may regard a strict $(2,2)$-category, i.e. a category enriched in small $(1,1)$-categories, as a $2$-category as follows: Taking nerves on mapping spaces, we obtain a category enriched in simplicial spaces from which we obtain a $2$-category using the equivalence of models for $(\infty,2)$-categories from \cite[Thm. 0.0.4]{lur092}. We will only make use of this embedding when we use the Gray tensor product of $(1,1)$-categories in the upcoming \Cref{def:square} and the only thing we need to know about this embedding is that the Gray tensor product of $\LinOrd{m}$ and $\LinOrd{n}$ agrees with the one used in \cite{hhln23b} which is ensured by \cite[Prop. 5.1.9]{hhln23b}.
\end{remark}

\begin{construction} \label{def:square} 
    Restricting the Yoneda embedding along the bicosimplicial object 
    \begin{align*} 
        \Delta \times \Delta & \longrightarrow \Catn{2} \\
        (\LinOrd{m},\LinOrd{n}) & \longmapsto \LinOrd{n} \Gray \LinOrd{m}
    \end{align*}
    sending $(\LinOrd{m},\LinOrd{n})$ to the Gray tensor product $ \LinOrd{n} \Gray \LinOrd{m}$ (a strict $(2,2)$-category),  
    induces the \emph{oplax square functor}:\begin{align*} 
        \Sq \colon \Catn{2} &\longrightarrow \DblCat \\
        \doublecat & \longmapsto \left( ( \LinOrd{m} ,\LinOrd{n}) \mapsto \Fun_{\Catn{2}}(\LinOrd{n} \Gray \LinOrd{m},\doublecat) \right) 
    \end{align*}
    Here, $\Sq(\doublecat)(n,-)$ is indeed complete Segal as the Gray tensor product preserves colimits in each variable separately. 
\end{construction}

Next we discuss orthogonal factorization systems, which we will simply call factorization system because there is no other notion of factorization systems studied in this paper.

\begin{definition}[Joyal]
    A category $\mathcal C$ together with two subcategories $\catL$ and $\mathcal \catR$ is called an \emph{factorization system} if 
    \begin{itemize}
        \item The subcategories $\catL$ and $\catR$ contain every equivalence and
        \item for any commutative square
        \[ \begin{tikzcd}[sep=huge]
            \bullet \ar[r] \ar[d,twoheadrightarrow]  & \bullet \ar[d,rightarrowtail] \\
            \bullet \ar[r] \ar[ur,dashed] & \bullet
        \end{tikzcd}\]
    \end{itemize}
    in which the left morphism is contained in $\catL$ and the right morphism is $\catR$, there is a unique dashed filler. 

    The morphisms in $\catL$ are called \emph{egressive}, the once in $\catR$ \emph{ingressive}.
    The category of factorization systems $\OFS$ is defined as the subcategory of the category of functors $ \Fun(\Lambda^2_2,\CatInfty)$ from the cospan $\Lambda^2_2 = (0 \to 2 \leftarrow 1)$ to $\CatInfty$ where $0 \to 2$ and $2 \leftarrow 1$ are sent to the inclusion of the class of egressive (and ingressive, respectively) morphisms of a factorization system. 
\end{definition}
\begin{remark}
    In \cite[Def. 5.2.8.8]{lur09}, Lurie also requires the two classes of morphisms two be closed under retracts. This condition is redundant by \cite[Sec. 1.1]{gkt18}.
\end{remark}
\begin{notation}
    Given two categories $\cat$ and $\catS$, there is a factorization system $\cat \prodOFS \catS$ on $\cat \times \catS$ with $(\cat \prodOFS \catS)_{\mathrm{eg}} = \cat \times \core{\catS}$ and $(\cat \prodOFS \catS)_{\mathrm{in}} = \core{\cat} \times \catS$.
\end{notation}
\begin{proposition}[{\cite[Prop. A.0.4]{bs24}}] \label{detectingOFS}
Let $\catL, \catR \subset \cat$ be subcategories containing every equivalence. The following are equivalent:
    \begin{itemize}
        \item The triple forms a factorization system.
        \item The restricted composition map 
        \begin{align*}
            \core{\Ar{\catL}} \times_{\core{\cat}} \core{\Ar{\catR}} &\longrightarrow \core{\Ar{\cat}} \\
            (f,g) & \longmapsto g \circ f
        \end{align*}is an equivalence.
    \end{itemize}
\end{proposition}

\section{Double categories and factorization systems} 
\label{section-mainthm}

In this section, we will first define a subcategory of \emph{factorization double categories} $\DblCatOF \subset \DblCat$ (\Cref{def-DblOF}). 
We will construct the functor \[ \DCLR \colon \OFS \longrightarrow \DblCatOF \] in \Cref{def-DCLR} and \Cref{lem-defDCLR}. Then, we construct its inverse functor \[ \corners \colon \DblCatOF \longrightarrow \OFS \] in \Cref{constr-corners}. 
Afterwards, we will explicitly define unit and counit transformations and check that they are equivalences, proving \Cref{main-theorem}.

\begin{proposition} \label{equ-def-DblOF}
    For $\doublecat$ a double category, the following are equivalent
    \begin{enumerate}
        \item The square \[ 
        \begin{tikzcd}[sep = huge]
            \doublecat(1,1) 
            \ar[r,"{\doublecat(\id,d_0)}"]  
            \ar[d,swap, "{\doublecat(d_1,\id)}"]& \doublecat(1,0) \ar[d,"{\doublecat(d_1,\id)}"] 
            \\
            \doublecat(0,1) 
            \ar[r,"{\doublecat(\id,d_0)}"] 
            & \doublecat(0,0)
        \end{tikzcd} \]
        is cartesian.
        \item The functor 
        \[ \doublecat(-,d_0) \colon \doublecat(-,1) \longrightarrow \doublecat(-,0)\] is a left fibration.
        \item The functor 
        \[ \doublecat(d_1,-) \colon \doublecat(1,-) \longrightarrow \doublecat(0,-) \] is a right fibration.
        \item The functor 
        \[ \doublecat(-,d_0) \colon \doublecat(-,n) \longrightarrow \doublecat(-,0)\] is a left fibration for all $n \ge 0$.
        \item \label{right-all-n} The functor  
        \[ \doublecat(d_n,-) \colon \doublecat(n,-) \longrightarrow \doublecat(0,-) \] is a right fibration for all $n \ge 0$.
    \end{enumerate}
\end{proposition}
\begin{proof}
    The first condition is the weakest. It implies the second and third condition by \Cref{rightfibrations}. 
    The second condition implies the fourth by evaluating at $n$ and using \Cref{rightfibrations}. Similarly, the third condition implies the fifth.
\end{proof}

\begin{definition} \label{def-DblOF}
    If a double category satisfies one of the equivalent conditions from \Cref{equ-def-DblOF}, we call it a \emph{factorization double category}. We write $\DblCatOF \subset \DblCat$ for the full subcategory spanned by the factorization double categories. 
\end{definition}

\begin{construction} \label{def-DCLR}
    Restricting the Yoneda embedding along the bicosimplicial object \begin{align*}
        \Delta \times \Delta & \longrightarrow \OFS \\
        (\LinOrd{m},\LinOrd{n}) & \longrightarrow [m] \prodOFS [n]
    \end{align*} 
    yields a functor \[ \DCLR \colon \OFS \longrightarrow \PSh(\Delta \times \Delta) \,. \] 
\end{construction}

\begin{lemma} \label{lem-defDCLR}
    The functor $\DCLR$ takes values in factorization double categories. 
\end{lemma}
\begin{proof}
    Let $\sOFS=(\cat,\catL,\catR)$ be a factorization system.
    The simplicial space $\DCLR(n,-)$ is a complete Segal space by \Cref{rezk-thm}: It arises as the Rezk nerve of the (non-full) subcategory of $\Fun(\LinOrd{n},\cat)$ spanned by the functors sending each morphism of $\LinOrd{n}$ to a morphism in $\catL$ and by the natural transformations which (pointwise) take values in $\catR$.
    Analogously, $\DCLR(-,n)$ is a complete Segal space.

    The category $\LinOrd{1} \times \LinOrd{1}$ is the pushout of two copies of $\LinOrd{2}$ along $\LinOrd{1}$. Mapping out of this pushout, passing to subspaces and also using that $\LinOrd{2}$ is a pushout of $\LinOrd{1}$ and $\LinOrd{1}$ along $\LinOrd{0}$ yields the following pullback:
    \[ \begin{tikzcd}[sep = huge]
        \map_{\OFS}(\LinOrd{1} \prodOFS \LinOrd{1},\sOFS) \ar[d,swap,"{(d_1 \times \id),(\id \times d_0)}"] \ar[r,"{(\id \times d_1),(d_0 \times \id)}"] & \core{\Ar{\catL}} \times_{\core{\cat}} \core{\Ar{\catR}} \ar[d,"\circ"] \\
    \core{\Ar{\catR}} \times_{\core{\cat}} \core{\Ar{\catL}} \ar[r,"\circ"] & \core{\Ar{\cat}}
    \end{tikzcd}  \]

    The right hand map is an equivalence by \Cref{detectingOFS}, implying that the left map is an equivalence too. Unwinding definitions, this shows that $\DCLR(\sOFS)$ is a factorization double category. 
\end{proof}
\begin{example} \label{product-OFS}
    Using that the projection $\LinOrd{k} \times \LinOrd{m} \to \LinOrd{k}$ is a localization at the class of morphisms which are constant in $\LinOrd{k}$, we find that for two categories $\mathcal C$ and $\mathcal D$, there is a natural equivalence \[ \DCLR( \cat \prodOFS \catS ) \cong \mathcal C \boxtimes \mathcal D \,. \]
\end{example}
\begin{remark}
    Actually, the above construction makes sense more generally: One can assign a double category to any category equipped with two subcategories containing every equivalence. The resulting double category is a factorization double category if and only if the subcategories form a factorization system.
\end{remark}
\begin{remark}
    It follows from \cite[Prop. 5.2]{ch21} together with \cite[Thm. 1.1]{rs22} that $\OFS$ is presentable. The functor \[ \DCLR \colon \OFS \longrightarrow \DblCat \] therefore admits a left adjoint. 
    However, this left adjoint uses colimits in the category of factorization systems, which are hard to compute. We will therefore not show directly that the counit of the adjuntion is an equivalence. 
    Instead, we will explicitly construct an inverse functor defined on the subcategory of factorization double categories. 
    It follows a posteriori that this inverse functor is the restriction of the left adjoint obtained by the adjoint functor theorem.
\end{remark}

We will now turn towards the definition of the inverse functor $\corners \colon \DblCatOF \to \OFS$.

\begin{construction} \label{defArDC}
    For $\cat$ a category, the arrow category $\Ar{\mathcal C}$ admits the structure of a factorization system with egressive morphisms the once which are an equivalence in the source and ingressive morphisms the once which are an equivalence in the target, see \cite[Exa. 4.7]{hhln23a}. We write \[ \ArDC{\cat} = \DCLR(\Ar{\cat})  \] for the \emph{double arrow category}.

    The combined source and target functor $(t,s) \colon \Ar{\cat} \to \cat \times \cat$ enhances to a functor of factorization systems to $\cat \prodOFS \cat$. This induces a functor
    \[ (t,s) \colon \ArDC{\cat} \longrightarrow \DCLR(\cat \prodOFS \cat) \cong \cat \boxproduct \cat \,. \] Postcomposing with the map $\cat \to \LinOrd{0}$ gives maps 
    \begin{align*}
        s \colon \ArDC{\cat} &\longrightarrow \LinOrd{0} \boxproduct \LinOrd{\cat} 
        \intertext{and}
        t \colon \ArDC{\cat} &\longrightarrow \LinOrd{\cat} \boxproduct \LinOrd{0} \,.
    \end{align*}
\end{construction}
\begin{remark}
    It follows from the comparison of thin and fat joins, \cite[Prop. 4.2.1.2]{lur09} that there is a natural map $\LinOrd{m} \times \LinOrd{n} \times \LinOrd{1} \longrightarrow \join{\LinOrd{n}}{\LinOrd{m}}$ which exhibits the target as the localization of the source at the morphisms which are either constantly $0$ in the $\LinOrd{1}$-coordinate and constant in the $\LinOrd{n}$-coordinate or constantly $1$ and constant in the $\LinOrd{m}$-coordinate. Unwinding definitions, we find that \[ \ArDC{\cat}(m,n) \cong \map(\join{\LinOrd{n}}{\LinOrd{m}},\cat) \,, \] which is the definition used in \cite{nui21}. 
\end{remark}

\begin{construction}
    Restricting the Yoneda embedding along the cosimplicial object 
    \begin{align*} \Delta &\longrightarrow \DblCat \\ 
     \LinOrd{n} & \longmapsto \ArDC{\LinOrd{n}}\end{align*}
     yields a functor:
     \begin{align*}
        \precorners \colon \DblCat &\longrightarrow \Fun(\Delta^{\op},\An) \\
        \doublecat & \longmapsto \left( n \mapsto \map_{\DblCat}(\ArDC{\LinOrd{n}},\doublecat) \right) 
     \end{align*}
\end{construction}

We now want to check that, when restricted to factorization double categories, this functor takes values in complete Segal spaces and that it enhances to a functor to factorization systems. 

\begin{construction}
    For fixed $n \ge 2$, let $P$ be the poset of subsets of $\LinOrd{n}=\{0, \dots , n \}$ which are either a singleton $\{i\}$ for $1 \le i \le n-1$ or a two-element subset containing two consecutive elements. 
    Using that all these subsets are linearly ordered posets and all the inclusions preserve the order, we obtain a functor $p \colon P \to \Delta$ and the inclusion of the subsets into $\LinOrd{n}$ induces a natural transformation $p \Rightarrow \const \LinOrd{n}$. 
    Let $\mathcal C$ be a category with finite colimits and let $X \colon \Delta \to \mathcal C$ be a cosimplicial object in $\cat$. We write $\genSpine[n]{X} = \colim_P (X \circ p)$ for the colimit of the composite functor. The natural transformation induces a morphism $\genSpine[n]{X} \to X(n)$, which we call the spine inclusion.

    If $X \colon \Delta \to \Fun(\Delta^{\op},\An)$ is the Yoneda embedding, we recover the inclusion of simplicial sets usually referred to as the \enquote{spine inclusion}. We will also simply denote this spine by $\spine{n}$. 
    In this case, we can also describe $I_n \to \LinOrd{n}$ more explicitly as the inclusion of the simplicial subset spanned by those $1$-simplices $\LinOrd{1} \to \LinOrd{n}$ which increase by at most $1$. The local objects with respect to those morphisms are precisely the Segal spaces. 
\end{construction}

\begin{lemma} \label{saturation-lemma}
    The saturation (under pushouts and retracts) of the following two sets of functors of double categories agree:
    \begin{itemize}
        \item the one morphism 
        \begin{equation} \label{OFSwithRLP} (d_1,\id),(\id,d_0) \colon \left(\LinOrd{0} \boxproduct \LinOrd{1} \right)\cup_{\LinOrd{0} \boxproduct \LinOrd{0}} \left(\LinOrd{1} \boxproduct \LinOrd{0}\right) \longrightarrow \LinOrd{1} \boxproduct \LinOrd{1} \,.\end{equation}
        \item the spine inclusions 
        \begin{equation} \label{SegalwithRLP}
            \genSpine[n]{\ArDC{-}} \longrightarrow \ArDC{n}
        \end{equation}
        for $n \ge 2$.
    \end{itemize}
\end{lemma}
\begin{proof}
    It is easy to see that the map \eqref{OFSwithRLP} is a retract of the map \eqref{SegalwithRLP} for $n = 2$. We leave the details to the reader, since we will not make use of this direction in the rest of the paper. 
    
    For the other direction, we will proceed as follows:
    \begin{itemize}
        \item We will define a certain bicosimplicial subspace $P \subset \ArDC{\LinOrd{n}}$ such that the inclusion is sent to an equivalence under the localization functor $\PSh(\Delta \times \Delta) \to \DblCat$. 
        \item We will show that $P$ can be obtained from the source of \eqref{SegalwithRLP} via iterated pushouts of \eqref{OFSwithRLP} in $\PSh(\Delta \times \Delta)$.
    \end{itemize}
    The claim follows because the localization $\PSh(\Delta \times \Delta) \to \DblCat$ preserves colimits. 
    
    Let $P \subset \ArDC{\LinOrd{n}}$ be the bicosimplicial subspace spanned by the elements \[ \ArDC{\LinOrd{n}}(k,m) = \map(\join{\LinOrd{m}}{\LinOrd{k}},\LinOrd{n}) \] for which every morphism in the image of the natural inclusions $\LinOrd{m} \to \join{\LinOrd{m}}{\LinOrd{k}}$ and $\LinOrd{k} \to \join{\LinOrd{m}}{\LinOrd{k}}$ is sent to a morphism in $\LinOrd{n}$ which increases by at most $1$. 
    Let $Q \subset \ArDC{\LinOrd{n}}$ be the bicosimplicial subspace spanned by all the elements where this holds just for the morphisms  in the image of the natural inclusions $\LinOrd{m} \to \join{\LinOrd{m}}{\LinOrd{k}}$.

    Let $\map^{\le 1}(\LinOrd{m},\LinOrd{n}) \subset \map(\LinOrd{m},\LinOrd{n})$ denote the subset of maps where each morphism in $\LinOrd{m}$ is sent to a morphism in $\LinOrd{n}$ which increases by at most $1$.
    
    Every map $\alpha \colon \LinOrd{m} \to \LinOrd{n}$ gives an element in \[ \DCLR(\Ar{\LinOrd{n}})(n-\alpha(m),m)=\map(\join{\LinOrd{m}}{\LinOrd{n-\alpha(m)}},\LinOrd{n}) \] whose restriction to $\LinOrd{m}$ is $\alpha$ and whose restriction to $\LinOrd{n-\alpha(m)}$ is the inclusion of the last $n-\alpha(m)+1$ elements.
    This induces equivalences
    \begin{align*}
        \ArDC{\LinOrd{n}}(-,m) &{}\cong \coprod_{\alpha \in \map(\LinOrd{m},\LinOrd{n})} \LinOrd{n-\alpha(m)}
        \intertext{and similarly} 
        \ArDC{\LinOrd{n}}(k,-) &{}\cong \coprod_{\alpha \in \map(\LinOrd{k},\LinOrd{n})} \LinOrd{\alpha(0)} \,,
        \intertext{inducing equivalences on simplicial subspaces}
        Q(-,m) &{}\cong \coprod_{\alpha \in \map^{\le 1}(\LinOrd{m},\LinOrd{n})} \LinOrd{n-\alpha(m)} \,, \\
        Q(k,-) &{}\cong \coprod_{\alpha \in \map(\LinOrd{k},\LinOrd{n})} \spine{\alpha(0)}
        \intertext{and}
        P(-,m) &{}\cong \coprod_{\alpha \in \map^{\le 1}(\LinOrd{m},\LinOrd{n})} \spine{n-\alpha(m)}
    \end{align*}

    From this description it follows that the inclusion $P \to Q$ exhibits the target as the localization of the source with respect to the reflective inclusion \[ \Fun(\Delta^{\op},\CatInfty) \longrightarrow \Fun(\Delta^{\op},\PSh(\Delta)) \cong \PSh(\Delta \times \Delta) \] and so is $Q \to \ArDC{ \LinOrd{n} }$ for the other coordinate.
    
    But both these inclusions are intermediate inclusions of $\DblCat \subset \PSh(\Delta \times \Delta)$, it follows that the inclusion $P \to \ArDC{\LinOrd{n}}$ is sent to an equivalence by the localization functor to double categories. 

    For $0 \le j \le i \le n-1$, let $P^{i,j}$ denote the bicosimplicial subspace of $P$ spanned by all objects $\alpha \colon \LinOrd{1} \to \LinOrd{n}$ of \[ P(0,0) = \ArDC{\LinOrd{n}}(0,0)= \map( \join{\LinOrd{0}}{\LinOrd{0}} ,\LinOrd{n}) \] such that $\alpha(1) - \alpha(0) \le 1$, or $\alpha(1) \le i$, or $\alpha(1)=i+1$ and $\alpha(0) \ge j$. We get a nested sequence 
    \[ P^{1,1} \subset P^{1,0}=P^{2,2} \subset P^{2,1} \subset P^{2,0}=P^{3,3} \subset \dots P^{n-1,0}=P \,, \] 
    see \Cref{figure} for an example. 
    Note that the inclusion $P^{1,1} \subset \ArDC{\LinOrd{n}}$ is equivalent to the spine inlusion $\genSpine[n]{\ArDC{-}} \longrightarrow \ArDC{n}$. This can be seen using that colimits in $\PSh(\Delta \times \Delta)$ are computed pointwise and the colimit used to define $\genSpine[n]{\ArDC{-}}$ gives a bicosimplicial space which already is a double category. 
    \begin{figure}[h] 
        \centering
        \begin{tikzcd}
            00 \ar[r] & 01 \ar[d] \ar[r] \ar[dr, phantom, "{P^{1,0}}"] & 02 \ar[r] \ar[d] \ar[dr, phantom, "{P^{2,0}}"] & 03 \ar[d]   \\
            & 11 \ar[r] & 12 \ar[d] \ar[r] \ar[dr, phantom, "{P^{2,1}}"] & 13 \ar[d] \\
            & & 22 \ar[r] & 23 \ar[d] \\
            & &  & 33 
        \end{tikzcd} 
        \caption{$\ArDC{\LinOrd{3}}$, the squares are labeled by the the smallest bicosimplicial subspace $P^{i,j}$ in which they are contained}
        \label{figure}
    \end{figure}
    
    Again using that pushouts in $\PSh(\Delta \times \Delta)$ are computed pointwise, one verifies that the following square is a pushout for all $1 \le j+1 \le i \le n-1$:
    \[ 
    \begin{tikzcd}
        \left(\LinOrd{0} \boxproduct \LinOrd{1} \right)\cup_{\LinOrd{0} \boxproduct \LinOrd{0}} \left(\LinOrd{1} \boxproduct \LinOrd{0}\right)  \ar[d,hook] \ar[r] & P^{i,j+1} \ar[d,hook] \\
        \LinOrd{1} \boxproduct \LinOrd{1} \ar[r] & P^{i,j}
    \end{tikzcd} \] 
    where the lower map is the map classifying the element in \[ P^{i,j}(1,1) \subset \ArDC{\LinOrd{n}}(1,1)  = \map(\join{\LinOrd{1}}{\LinOrd{1}},\LinOrd{n}) \] restricting to the inclusion of $\{j,j+1\}$ on the first copy of $\LinOrd{1}$ and the inclusion of $\{i,i+1\}$ in the second.
    In fact, in this case the diagram above is pointwise a diagram of sets. Moreover, an explicit combinatorial argument shows that pointwise all maps in the above square are mapped to inclusions of sets and the upper left corner is precisely the intersection of the two inclusion into the lower right corner. We leave the details to the reader.

    This finishes the proof of the claim that $P$ can be obtained from $P^{1,1}$ via iterated pushouts of \eqref{OFSwithRLP}. 
\end{proof}

\begin{proposition} \label{precornesOFS}
    The simplicial space $\precorners(\doublecat)$ is a complete Segal space if and only if $\doublecat$ is a factorization double category. 
\end{proposition}
\begin{proof}
    A double category $\doublecat$ is a factorization category if it is local with respect to the morphism \eqref{OFSwithRLP}. Unwinding definitions, $\precorners(\doublecat)$ is a Segal space if and only if $\doublecat$ is local with respect to the morphisms \eqref{SegalwithRLP}. Therefore, the Segal claim follows from \Cref{saturation-lemma}. It is left to show that in this case the Segal space is automatically complete. We will use the criterion from \Cref{checkcompleteness}.
    
    First note that degeneracy map 
    \[ \precorners(\doublecat)_0 = \doublecat(0,0) \cong \doublecat(0,0) \times_{\doublecat(0,0)} \doublecat(0,0) \longrightarrow \doublecat(1,0) \times_{\doublecat(0,0)} \doublecat(0,1) = \precorners(\doublecat)_1 \] 
    is an inclusion of path components because it is a pullback of such.

    We need to show that every equivalence is contained in the essential image, which consists of the objects in $\doublecat(1,0) \times_{\doublecat(0,0)} \doublecat(0,1) = \precorners(\doublecat)_1$ which are equivalences in both components. 

    Let $c,d \in \doublecat(0,0) = \precorners(\doublecat)_0$ be two objects and \[ (f,g) \in \core{\left(\underslice{\doublecat(-,0)}{c}\right)} \times_{\doublecat(0,0)} \core{\left(\overslice{\doublecat(0,-)}{d} \right)}= \map_{\precorners(\doublecat)}(c,d) \,. \] be an equivalence.

    Then postcomposition with $(f,g)$ induces an equivalence \[ \map_{\precorners(\doublecat)}(d,c) \longrightarrow \map_{\precorners(\doublecat)}(d,d) \] Picking an inverse of the identity and unwinding composition in $\precorners(\doublecat)$, we see that $g$ has a left inverse. Similarly, one can show that it has a left inverse and so does $f$. 
\end{proof}
\begin{construction} \label{constr-corners}
    \Cref{precornesOFS} together with \Cref{rezk-thm} implies that we obtain a functor 
    \begin{align*}
        \precorners \colon & \DblCatOF \longrightarrow \CatInfty \,. 
        \intertext{We now refine this to a functor}
        \corners \colon&  \DblCatOF \longrightarrow \OFS \,, 
    \end{align*}
    which assigns to a factorization double category its \emph{category of corners}.

    The target functor 
    \[  \ArDC{\LinOrd{m}} \longrightarrow \LinOrd{m} \boxproduct \LinOrd{0} \] 
    gives rise to a natural functor \[ \doublecat(-,0) \longrightarrow \precorners(\doublecat) \] 
    which induces an equivalence on cores and an inclusion of path components 
    \[ \doublecat(1,0) \cong \doublecat(1,0) \times_{\doublecat(0,0)} \doublecat(0,0) \longrightarrow \doublecat(1,0) \times_{\doublecat(0,0)} \doublecat(0,1) = \core{\Ar{\precorners(\doublecat)}} \] 
    on cores of arrow categories.

    Similarly, we obtain a natural functor $\doublecat(0,-) \to \precorners(\doublecat)$. 
    The triple \[ \corners(\doublecat) = (\precorners(\doublecat),\doublecat(-,0),\doublecat(0,-)) \] forms a factorization system by \Cref{detectingOFS}.
\end{construction}

\begin{construction} \label{construction_unit}
    There are natural equivalences 
    \[ \map(\LinOrd{k},\LinOrd{m} \times \LinOrd{n}) \cong \map(\LinOrd{k},\LinOrd{m}) \times \map(\LinOrd{k},\LinOrd{n}) \cong \map(\LinOrd{k} \boxproduct \LinOrd{k},\LinOrd{m} \boxproduct \LinOrd{n})  \] 
    Precomposition with $(t,s) \colon \ArDC{\LinOrd{k}} \to \LinOrd{k} \boxproduct \LinOrd{k}$ from \Cref{defArDC} hence induces a natural functor
    \[ \LinOrd{m} \times \LinOrd{n} \longrightarrow \precorners(\LinOrd{m} \boxproduct \LinOrd{n}) \] 
    and this functor enhances to a functor of factorization systems
    \[ \LinOrd{m} \prodOFS \LinOrd{n} \longrightarrow \corners(\LinOrd{m} \boxproduct \LinOrd{n}) \,. \] 
    We hence obtain an element in 
    \[ \DCLR(\corners(\LinOrd{m} \boxproduct \LinOrd{n}))(m,n) = \map_{\DblCat}(\LinOrd{m} \boxproduct \LinOrd{n},\DCLR(\corners(\LinOrd{m} \boxproduct \LinOrd{n}))) \,. \] 
    Applying $\DCLR \circ \corners$ and precomposition with this morphism gives a natural map
    \begin{equation} \label{unitpointwise}
        \map_{\DblCat}(\LinOrd{m} \boxproduct \LinOrd{n},\doublecat) \longrightarrow \map_{\DblCat}(\LinOrd{m} \boxproduct \LinOrd{n},\DCLR(\corners(\doublecat))) 
    \end{equation}
    and hence a natural transformation \[ \id_{\DblCatOF} \Longrightarrow \DCLR \circ \corners \,. \]
\end{construction}

\begin{proposition} \label{unitisiso}
    The above natural transformation is a natural equivalence.
\end{proposition}
\begin{proof}
    We need to prove that \eqref{unitpointwise} is an equivalence for all $m,n$ and $\doublecat$. By the Segal condition, it is enough to check this for $m,n \le 1$. Because both sides are factorization double categories, we can also exclude the case $m=n=1$. In the remaining three cases, one verifies that the map is an equivalence by unwinding the definitions, e.g. in the case $m=1, n=0$ the map identifies with the natural map from $\doublecat(1,0)$ into the core of the $ \Ar{\corners(\doublecat)_{\mathrm{eg}}}$, which was defined to be $\doublecat(1,0)$.
    \end{proof}
\begin{construction} \label{construction-unit}
    Let $\sOFS$ be a factorization system on a category $\cat$. 
    There is a natural map 
    \begin{align*} \label{universal-prop-Ar} \map_{\OFS}\left(\Ar{\LinOrd{n}},\sOFS\right) \longrightarrow \map_{\CatInfty}\left(\LinOrd{n},\cat\right) \end{align*}
    given by passing to the functor of underlying categories and then embedding $\LinOrd{n}$ into $\Ar{\LinOrd{n}}$ by sending each element to the identity arrow. 
    It follows from \cite[Prop. 4.8]{hhln23a} that this map is an equivalence. 
    The result is stated for orthogonal adequate triples but the proof makes no essential use of the adequate triple property. 
    
    Applying the functor $\DCLR$ yields a map 
    \[  \map_{\OFS}\left(\Ar{\LinOrd{n}},\sOFS\right) \longrightarrow \map_{\DblCat}\left(\ArDC{\LinOrd{n}},\DCLR\left(\sOFS\right)\right) \,. \] 
    Combining these two map yields a natural functor
    \[ \mathcal C \longrightarrow \precorners\left(\DCLR\left(\sOFS\right)\right) \,. \]
\end{construction}
\begin{proposition}\label{counitisiso}
    The above natural transformation refines to a natural transformation of automorphisms of the category of factorization systems \[ \id_{\OFS} \Longrightarrow \corners \circ \DCLR \] and this natural transformation is a natural equivalence.
\end{proposition}
\begin{proof}
    Consider the following commutative diagram 
    \[ \begin{tikzcd}
        \map_{\OFS}\left((\LinOrd{1},\LinOrd{1},\core{\LinOrd{1}}),\sOFS\right) \ar[r,"\DCLR"] \ar[d,"{t^\ast}"] & \map_{\DblCat}\left(\LinOrd{1} \boxproduct \LinOrd{0},\DCLR\left(\sOFS\right)\right) \ar[d,"{t^\ast}"] \\
        \map_{\OFS}\left(\Ar{\LinOrd{1}},\sOFS\right) \ar[r,"\DCLR"] & \map_{\DblCat}\left(\ArDC{\LinOrd{1}},\DCLR\left(\sOFS\right)\right)
    \end{tikzcd} \,. \]
    The bottom arrow is the one used \Cref{construction-unit} to define the functor on cores of arrow categories $\core{\Ar{\sOFS}} \to \core{\Ar{\corners(\DCLR(\sOFS))}}$, and the vertical arrows induce the inclusion of the class of egressive morphisms in $\sOFS$ and $\corners(\DCLR(\sOFS))$, respectively. The top arrow is an equivalence, it can be identified with the identity functor on the subspace of $\map_{\CatInfty}(\LinOrd{1},\cat)$ spanned by the functors sending the morphism to an egressive morphism. 

    This shows that the functor in question preserves the class of egressive morphisms. 
    Analogously, one proves that it preserves ingressive morphisms. The argument actually shows that it induces an equivalence on the space of the egressive (or ingressive, respectively) morphisms. 
    Moreover, it follows from the definitions that it induces an equivalences on cores. 
    It follows from \Cref{detectingOFS} that the functor is an equivalence of factorization systems. 
\end{proof}
We can now deduce the main theorem:
\begin{theorem} \label{main-theorem}
    The functor $\DCLR$ induces an equivalence of categories \[ \OFS \cong \DblCatOF \]
\end{theorem}
\begin{proof}
    This follows from \Cref{unitisiso} and \Cref{counitisiso}. 
\end{proof}
This is an $\infty$-categorical
analog of \cite[Theorem 3.7]{ms23}.
\section{Fibrations} \label{section-fibrations}
In this section, we will first recall the definition of \emph{op-Gray fibrations} (and \emph{curved orthofibrations}) of factorization systems as well as \emph{\cartright s} (and \emph{\cocartright s}) of double categories. 
We will then verify in \Cref{comparefibrations} that the two notions agree under the equivalence between factorization systems and factorization double categories from the previous section.
In \Cref{maintheorem}, we will prove an (un)straightening equivalence for those fibrations.

\begin{definition}[{\cite[Def. 5.14]{hhln23a}}]
    A functor $F \colon \sOFSS \to \sOFS$ of factorization systems is called an \emph{ingressive cartesian fibration} if ingressive morphisms admit $F$-cartesian lifts and those lifts prescisely make up the subcategory of ingressive morphisms in $\sOFSS$. 

    A functor $F \colon \sOFSS \to \sOFS$ of factorization systems is called 
    \begin{itemize}
        \item a \emph{curved orthofibration} if it is an ingressive cartesian fibration and $F_{\mathrm{eg}} \colon \catSL \to \catL$ is a cocartesian fibration and
        \item an \emph{op-Gray fibration} if it is an ingressive cartesian fibration and $F_{\mathrm{eg}} \colon \catSL \to \catL$ is a cartesian fibration.
    \end{itemize}
    We write $\Ortho(\sOFS)$ for the non-full subcategory of the large category $\OFS_{/\sOFS}$ spanned by the curved orthofibrations over $\sOFS$ and functors over $\sOFS$ preserving cocartesian egressive lifts. Similarly, we define the large category $\opGray(\sOFS)$ of curved op-Gray fibrations. 

    As the category of orthogonal factorization systems $\OFS$ has pullbacks which commute with the forgetful functor to $\CatInfty$ by \cite[Prop. 5.2]{ch21} and both types of fibrations are preserved by pullbacks, \cite[Obs. 5.2(3)]{hhln23a}, we obtain functors 
    \[ \opGray,\, \Ortho \colon \OFS^{\op} \longrightarrow \CatInfty \,. \]
\end{definition}
\begin{lemma} \label{equivalent-condition-ing-cart}
    A functor $F \colon \sOFSS \to \sOFS$ is an ingressive cartesian fibration if and only if $F_{\mathrm{in}} \colon \catSR \to \catR$ is a right fibration.  
\end{lemma}
\begin{proof}
    It is observed in \cite[Obs. 5.2(2)]{hhln23a} that an ingressive morphism in $\catS$ is $F$-cartesian if and only if it is $F_{\mathrm{in}}$-cartesian (while not making any use of the additional assumption that the factorization systems are adequate). It follows that $F$ is an ingressive cartesian fibration if and only if $F_{\mathrm{in}} \colon \catSR \to \catR$ is a cartesian fibration for which every morphisms in the source is $F_{\mathrm{in}}$-cartesian. But that is equivalent to $F_{\mathrm{in}}$ being a right fibration. 
\end{proof}
\begin{definition}
    A functor $F \colon \doublecatS \to \doublecat$ of double categories is called a \emph{\cocartright} if $F(-,0)$ is a cocartesian fibration and $F(n,-)$ is a right fibration for all $n \ge 0$.

    Let $\CocartRight{\doublecat}$ denote the non-full subcategory of the large category $\overslice{\DblCat}{\doublecat}$ spanned by the \cocartright s and functors preserving cocartesian edges.
    We analogously define the category $\CartRight{\doublecat}$ of \emph{\cartright s} over $\doublecat$. 
    
    Cocartesian fibrations, functors preserving cocartesian edges and right fibrations are preserved by pullbacks. Therefore,
    pullbacks also preserve \cocartright s between double categories and functors between them, we therefore obtain functors
    \[ \operatorname{CoR}, \, \operatorname{CaR} \colon \DblCat^{\op} \longrightarrow \CatInfty \,. \]
\end{definition}
\begin{lemma} \label{source-of-cartright}
    Let $F \colon \doublecatS \to \doublecat$ be a functor of double categories such that $F(0,-)$ is a right fibration. Then $F(n,-)$ is a right fibration for all $n \ge 0$ if and only if $\doublecatS$ is a factorization double category. 
\end{lemma}
\begin{proof}
    We use characterization \ref{right-all-n} from \Cref{equ-def-DblOF} for factorization double categories. 
    Consider the following commutative diagram: 
    \[ \begin{tikzcd}[sep=huge]
        \doublecatS(n,-) \ar[d,swap,"{\doublecatS(d_n,-)}"] \ar[r,"{F(n,-)}"] & \doublecat(n,-) \ar[d,"{\doublecat(d_n,-)}"] \\
        \doublecatS(0,-) \ar[r,swap,"{F(0,-)}"] & \doublecat(0,-)
    \end{tikzcd} \]
    It follows from our assumption that the bottom functor and the functor on the right are right fibrations. 
    Using left cancellation of right fibrations, we deduce that the upper functor is a right fibration for all $n \ge 0$ if and only if the functor on the left is a right fibration for all $n \ge 0$.
\end{proof}
\begin{proposition} \label{comparefibrations}
    Let $\sOFS$ be a factorization system. The functor $\DCLR \colon \OFS \to \DblCat$ induces equivalences 
    \begin{align*}
        \Ortho(\sOFS) &{} \cong \CocartRight{\DCLR\left(\sOFS\right)}
        \intertext{and} 
        \opGray(\sOFS) &{}\cong \CartRight{\DCLR\left(\sOFS\right)}  \,.
    \end{align*}
\end{proposition}

\begin{proof}
    We discuss the first equivalence, the other one can be proven analogously. 

    It follows from \Cref{main-theorem}, that $\DCLR$ induces a fully faithful inclusion \[ \overslice{\OFS}{\sOFS} \longhookrightarrow \overslice{\DblCat}{\DCLR{\sOFS}} \,. \]
    We need to check that this equivalence restricts to an equivalence on the subcategories $\Ortho(\sOFS)$ and $\CocartRight{\DCLR(\sOFS)}$. 
    
    Let $F \colon \sOFSS \to \sOFS$ be a functor of factorization systems. Unwinding definitions, we see that $F_{\mathrm{eg}}$ is naturally equivalent to $\DCLR(F)(-,0)$ and $F_{\mathrm{in}}$ is naturally equivalent to $\DCLR(F)(0,-)$. 
    It hence follows from \Cref{equivalent-condition-ing-cart} that $F$ is an ingressive cartesian fibration if and only if $\DCLR(F)(0,-)$ is a right fibration. But \Cref{source-of-cartright} implies that this automatically forces $\DCLR(F)(n,-)$ to be a right fibration for all $n \ge 0$.
    We conclude that $\DCLR(F)$ is a \cocartright{} if and only $F$ is a curved orthofibration and that a functor between two curved othofibrations preserves cocartesian lifts of egressive morphisms if and only if the associated functor of \cocartright s preserves cocartesian lifts. 
    All in all, we have shown that $\DCLR$ induces a fully faithful functor $\Ortho(\sOFS) \hookrightarrow \CocartRight{\DCLR(\sOFS)}$. 
    But this functor is also essentially surjective because the source of any \cartright{} over a factorization double category is a factorization double category by \Cref{source-of-cartright}.    
\end{proof}

We will now state the (un)straightening equivalence for \cocartright s. We will denote the $2$-category of categories, functors and natural transformations by $\CatInftyOplax$, for example defined in \cite[Def. 5.1.6]{hhln23b}.

\begin{theorem} \label{maintheorem}
    There is a natural equivalence \[ \core{\CocartRight{-}} \cong \map_{\DblCat}\left(\verop{(-)},\Sq\left(\CatInftyOplax\right)\right)  \] of functors from double categories to spaces. 
\end{theorem}

\begin{remark}
    Since $\CatInftyOplax$ is a large category, the mapping space actually has to be taken in the very large category of large categories. We omit this distinction from the notation.
\end{remark}

The proof strategy follows the one from \cite[Thm. 1.26]{af20}. We are grateful to Jaco Ruit for pointing us to this reference.

\begin{proposition} \label{rightKan}
    The functor \[ \core{\CocartRight{-}} \colon \DblCat^{\op} \longrightarrow \An \] is right Kan extended from its restriction along $\Delta^{\op} \times \Delta^{\op} \to \DblCat^{\op}$.
\end{proposition}
\begin{proof}
    Let $\doublecat$ be a double category. Using the pointwise formula for right Kan extension, we must show that the upper map in the following commutative diagram is an equivalence
    \[ \begin{tikzcd}
        \core{\CocartRight{\doublecat}} \ar[r] \ar[d] & \lim \left( \left( \overslice{\Delta \times \Delta}{\doublecat} \right)^{\op} \to \Delta^{\op} \times \Delta^{\op} \to \DblCat \to \An \right) \ar[d] \\
        \core{\left(\overslice{\PSh(\Delta \times \Delta)}{\doublecat}\right)} \ar[r] & \lim \left( \left( \overslice{\Delta \times \Delta}{\doublecat} \right)^{\op} \to \Delta^{\op} \times \Delta^{\op} \to \PSh(\Delta \times \Delta) \to \An \right)
    \end{tikzcd} \]
    where the last functor in the limit in the upper right corner is $\core{\CocartRight{-}}$ and on the lower right corner it is $\core{\overslice{\DblCat}{-}}$. 
    The left arrow is an inclusion of path components by definition and the right arrow is an inclusion of path components because it is a limit of such.
    The button arrow is an equivalence because $\PSh(\Delta \times \Delta)$ is a topos and the slice functor $\PSh(\Delta \times \Delta)^{\op} \longrightarrow \CatInfty$ hence does preserve limits by \cite[Thm. 6.1.3.9, Prop. 6.1.3.10]{lur09}.
    The claim hence follows from \Cref{helplemmadetectingfibrations}.
\end{proof}

\begin{lemma} \label{helplemmadetectingfibrations}
    Let $\doublecat$ be a double category and $F \colon \doublecatS \to \doublecat$ be a morphism of bisimplicial spaces such that the pullback along any map $\LinOrd{m} \boxtimes \LinOrd{n} \to \doublecat$ is a \cocartright \, of double categories. Then $F$ is a \cocartright \, of double categories. 
\end{lemma}
\begin{proof}
    Let us first check that $\doublecatS$ is a double category. 
    The large category of double categories is the category of local objects in $\PSh(\Delta \times \Delta)$ with respect to certain morphisms $\mathbb{E} \to \mathbb{E}^\prime$ (corepresenting the Segal and completeness condition), i.e. we need to check that for those morphisms, the induced map \[ \map_{\PSh(\Delta \times \Delta)}(\mathbb{E}^\prime,\mathbb{D}) \longrightarrow \map_{\PSh(\Delta \times \Delta)}(\mathbb{E},\mathbb{D})  \] is an equivalence. 
    This can equivalently be checked on fibers over the same morphism for $\doublecat$ (where we know that it is an equivalence), i.e. that 
    \begin{equation} \label{equ-fiberwise}
    \ast \times_{\map(\mathbb{E}^\prime,\doublecat)} \map(\mathbb{E}^\prime,\doublecatS)\longrightarrow \ast \times_{\map(\mathbb{E},\doublecat)} \map(\mathbb{E},\doublecatS) \end{equation} 
    is an equivalence for every point in $\map(\mathbb{E^\prime},\mathbb{C})$. 
    Now we use that 
    \begin{align*} 
        \ast \times_{\map(\mathbb{E}^\prime,\doublecat)} \map(\mathbb{E}^\prime,\doublecatS) &\cong \ast \times_{\map(\mathbb{E}^\prime,\mathbb{E}^\prime)} \map(\mathbb{E}^\prime,\doublecatS \times_{\doublecat} \mathbb{E}^\prime) 
        \intertext{and}
        \ast \times_{\map(\mathbb{E},\doublecat)} \map(\mathbb{E},\doublecatS) &\cong \ast \times_{\map(\mathbb{E},\mathbb{E}^\prime)} \map(\mathbb{E},\doublecatS \times_{\doublecat} \mathbb{E}^\prime)
    \end{align*} where the fiber on the right side is taken over the identity and the morphism $\mathbb{E} \to \mathbb{E}^\prime$, respectively. This shows that \eqref{equ-fiberwise} is an equivalence for $\doublecatS \to \doublecat$ if it holds when replacing the functor with the pullback along an arbitrary map $\mathbb{E}^\prime \to \doublecat$.

    In our concrete situation, we have that $\mathbb{E}^\prime= \LinOrd{m} \boxtimes \LinOrd{n}$ for any morphism appearing in the definition of a double category. But we did assume that the total space is a double category when being pulled back along a map from $\LinOrd{m} \boxtimes \LinOrd{n}$.

    Now we check that $F$ is a \cocartright . Let $\LinOrd{2} \to \doublecat(-,0)$ be an arbitrary functor. This functor is also represented by a functor  $\LinOrd{2} \boxtimes \LinOrd{0} \to\doublecat$. The pullback of $F$ along this map is a \cocartright. Evaluating at $(-,0)$, we see that the pullback of $F(-,0)$ along the auxiliary map $\LinOrd{2} \to \doublecat(-,0)$ is a cocartesian fibration. It follows from \cite[Prop. 2.23(1)(c)]{af20} that $F(-,0)$ is a cocartesian fibration. Similarly, one shows that $F(n,-)$ is a right fibration by pulling back along maps out of $\LinOrd{n} \boxtimes \LinOrd{2}$.
\end{proof}

\begin{proof}[Proof of \Cref{maintheorem}]

    Both functors are right Kan extended from their restriction along $\Delta^{\op} \times \Delta^{\op} \to \DblCat^{\op}$ by \Cref{rightKan} and because representable functors of reflective subcategories of presheaves are right Kan extended.

    Therefore, it is enough to prove that the functors are equivalent when restricted along $\Delta^{\op} \times \Delta^{\op} \to \DblCat^{\op}$. There will be a unique equivalence extending the equivalence on this restriction. 

    By definition, we have natural equivalences
    \[ \map_{\DblCat} \left(\verop{(\LinOrd{m} \boxproduct \LinOrd{n})},\Sq\left(\CatInftyOplax\right)\right) \cong \map_{\Catn{2}} \left(\LinOrd{n}^{\op} \Gray \LinOrd{m} , \CatInftyOplax \right)\,. \]
    Combining \cite[Prop. 5.2.10]{hhln23b}, \cite[Cor. 6.5]{hhln23a} and \cite[Thm. 2.5.1]{hhln23b}, there are also natural equivalences
    \[ \map \left( \LinOrd{n}^{\op} \Gray \LinOrd{m} ,\CatInftyOplax \right) \cong \Ortho(\LinOrd{m} \prodOFS \LinOrd{n}) \,. \] The claim hence follows from \Cref{comparefibrations} and \Cref{product-OFS}.
\end{proof}
    A functor $F \colon \doublecatS \to \doublecat$ of double categories is called a \leftcart{} if $\fullop{F}$ is a \cocartright. We denote the large category of \leftcart s over $\doublecat$ by $\LeftCart{\doublecat}$. 

    The following result, combined with \Cref{maintheorem}, establishes an equivalence of spaces of \cocartright s and \leftcart s. In \cite[Thm. 3.1]{nui21}, it is shown that this enhances to an equivalence of categories.
\begin{corollary} \label{fullop-maintheorem}
    There is a natural equivalence \[ \core{\LeftCart{-}} \cong \map_{\DblCat}\left(\verop{(-)},\Sq\left(\CatInftyOplax\right)\right)  \] of functors from factorization systems to spaces. 
\end{corollary}
\begin{proof}
    By \Cref{maintheorem}, we have natural equivalences 
    \begin{align*}
        \core{\LeftCart{\doublecat}} \cong{} & \core{\CocartRight{ \fullop{\doublecat} }} \\
        \cong{} & \map_{\DblCat}\left(\verop{\left(\fullop{\doublecat}\right)},\Sq\left(\CatInftyOplax\right)\right) \\
        \cong{} & \map_{\DblCat}\left(\verop{\doublecat},\swapop{\left( \Sq\left(\CatInftyOplax\right)\right)} \right) 
    \end{align*}
    because $\swapop{\left(\verop{\left(\fullop{\doublecat}\right)}\right)}=\verop{\doublecat}$. 
    
    Taking opposite categories induces an equivalence \[ \CatInftyOplax \cong \left(\CatInftyOplax \right)^{\mathrm{co}} \] where $\mathrm{co}$ denotes the operation of taking opposites of $2$-morphisms (see e.g. \cite[Rem. 3.1.10]{hhln23b}).

    Unwinding definitions and using that $\LinOrd{n} \Gray \LinOrd{m} \cong \left( \LinOrd{m} \Gray \LinOrd{n} \right)^{\mathrm{co}}$ (see e.g. \cite[Obs. 2.2.10]{agh24}), we obtain an induced equivalence \[ \Sq\left( \CatInftyOplax \right) \cong \swapop{ \Sq\left( \CatInftyOplax \right) } \,. \]
    Combining the two above facts, the claim follows. 
\end{proof}

\section{Adequate triples and span categories} \label{section-adequate}
In this section, we will recall the definition of an \emph{adequate factorization system} and classify when a factorization double category comes from an adequate factorization system (\Cref{main-comparison-Ad}). 
We will use this to prove that the category of adequate factorization systems has a unique non-trivial automorphism, the span category (\Cref{Aut(AOFS)}, \Cref{span=verop}).
\begin{definition}[{\cite[Def. 4.2]{hhln23a}}]
    Let $\sOFS$ be a factorization system. We call a square 
    \[ \begin{tikzcd}
        \bullet \ar[d,rightarrowtail] \ar[r,twoheadrightarrow] & \bullet \ar[d,rightarrowtail]  \\
        \bullet  \ar[r,twoheadrightarrow]  & \bullet
        \end{tikzcd} \]
    \emph{ambigressive} if the horizontal morphisms are egressive and the vertical morphisms are ingressive. Similarly, we call a a cospan 
    \[ \begin{tikzcd}
         & \bullet \ar[d,rightarrowtail]  \\
        \bullet  \ar[r,twoheadrightarrow]  & \bullet
    \end{tikzcd}\]
    \emph{ambigressive} if one leg is egressive and the other one is ingressive. 

    The factorization system $\sOFS$ is called \emph{adequate} if every ambigressive square is a pullback and if every ambigressive cospan admits a pullback. 

    We write $\OFSAd \subset \OFS$ for the full subcategory of adequate factorization systems.
\end{definition}

\begin{lemma} \label{lem-eq-adequate}
    A factorization system is adequate if and only if every ambigressive cospan can uniquely be extended to an ambigressive square. 
\end{lemma}
\begin{proof}
    The \enquote{only if}-part follows from the uniqueness of pullbacks. 
    
    Now assume that every ambigressive cospan can uniquely be extended to an ambigressive square. We must show that this square is a pullback. Equivalently, we must check that the space of fillers of the following diagram is contractible
    \[ 
    \begin{tikzcd}[sep=huge]
        t \ar[dr,dashed] \ar[d,twoheadrightarrow,swap,"f"] \ar[drr,bend left] &  \\
        s \ar[dr,rightarrowtail] & d \ar[d,rightarrowtail] \ar[r,twoheadrightarrow] & b \ar[d,rightarrowtail]  \\
        & c  \ar[r,twoheadrightarrow]  & a
    \end{tikzcd} \,,
    \] (here we already factored an arbitray morphism from $t$ to $c$ into an egressive morphism $f$ followed by an ingressive morphism). We can apply the lifting criterion for factorization systems to the left square, to see that the dashed arrow would admit a unique factorization through $f$. We might therefore reduce to the the situation where $f$ is an equivalence, i.e. 
    \[
    \begin{tikzcd}[sep=huge]
        t \ar[dr,dashed] \ar[r,twoheadrightarrow,"f^\prime"] \ar[ddr,rightarrowtail,bend right] & s^\prime \ar[dr,rightarrowtail,"g"] \\
         & d \ar[d,rightarrowtail] \ar[r,twoheadrightarrow,"h"] & b \ar[d,rightarrowtail,"j"]  \\
        & c  \ar[r,twoheadrightarrow,"k"]  & a
    \end{tikzcd} \,,
    \] (where we now factored the morphism from $t$ to $b$).
    The span $(d \xrightarrow{h} b \xleftarrow{g} s^\prime)$ can be extended uniquely to an ambigressive square. This in particular gives a morphism from $t$ to $d$. The composition of the lower square with the upper square gives an extension of the ambigressive cospan $(k \xrightarrow{k} a \xleftarrow{j \circ g} s^\prime)$. But the outer square already provided such an extension, which we assumed to be unique. This provides the necessary homotopies, making the diagram commute.
\end{proof}

\begin{definition} 
   A factorization double category $\doublecat$ is called \emph{adequate} if $\horop{\doublecat}$ is a factorization double category. We denote the full subcategory of adequate factorization double categories by $\DblCatAd \subset \DblCatAd$.
\end{definition}

\begin{proposition} \label{main-comparison-Ad}
    A factorization system $\sOFS$ is adequate if and only if $\DCLR(\sOFS)$ is adequate, i.e. the equivalence from \Cref{main-theorem} restricts to an equivalence \[  \OFSAd \cong \DblCatAd \,. \]
\end{proposition}
\begin{proof}
    Unwinding definitions, we find that $\DCLR(\sOFS)$ is adequate if and only if $\sOFS$ fulfills the condition from \Cref{lem-eq-adequate}, which is equivalent to being adequate.
\end{proof} 

We will now verify that $\horop{(-)}$ recovers the span construction from \cite[Def. 2.12, Prop. 4.9]{hhln23a} under the equivalence from \Cref{main-comparison-Ad}. 

We will actually compute the entire automorphism group of $\DblCatAd$, hence also generalizing \cite[Thm. 5.21]{hhln23a}, saying that the span category construction refines to a $\mathbb Z/2\mathbb{Z}$-action on $\OFSAd \cong \DblCatAd$.

\begin{theorem} \label{Aut(AOFS)}
    There is an equivalence of groups
    \[
        \Aut\left(\DblCatAd\right) \cong \mathbb{Z}/2\mathbb{Z}
    \]
    with generator $\horop{(-)}$ on the left side.
\end{theorem}
\begin{proof}
    Note that the image of the Yoneda embedding $\Delta \times \Delta \to \PSh(\Delta \times \Delta)$ is contained in $\OFSAd$. We first prove that any equivalence restricts to an equivalence of this subcategory.

    Any such equivalence $F$ must preserve the terminal object $\LinOrd{0} \boxproduct \LinOrd{0}$. From fully-faithfulness it follows that $F$ sends $\LinOrd{0} \boxproduct \LinOrd{1}$ and $\LinOrd{1} \boxproduct \LinOrd{0}$ to double categories with two objects. 

    There must be a vertical or horizontal morphism between these two objects because we could otherwise write the double category as a disjoint union of two non-initial double categories, which is not the case for $\LinOrd{0} \boxproduct \LinOrd{1}$ and $\LinOrd{1} \boxproduct \LinOrd{0}$. 
    
    We hence obtain a morphism 
    \begin{equation} \label{morphism-equivalence}
         \doublecat \longrightarrow F(\LinOrd{1} \boxproduct \LinOrd{0})
    \end{equation} which is an equivalence on objects where $\doublecat$ is either $\LinOrd{1} \boxproduct \LinOrd{0}$ or $\LinOrd{0} \boxproduct \LinOrd{1}$. Applying the same argument to $F^{-1}$ we obtain a morphism 
    \[
        \doublecatS \longrightarrow F^{-1}(\doublecat)\]
    which is an equivalence on objects where $\doublecatS$ is either $\LinOrd{1} \boxproduct \LinOrd{0}$ or $\LinOrd{0} \boxproduct \LinOrd{1}$.
    Applying $F$ to this morphism yields a morphism 
    \begin{equation}  \label{morphism-equivalence-2}
        F(\doublecatS) \longrightarrow \doublecat
    \end{equation}
    which is still an equivalence on objects because $F$ preserves $\LinOrd{0} \boxproduct \LinOrd{0}$ and is fully faithful. 
    Composing \eqref{morphism-equivalence-2} with \eqref{morphism-equivalence}, yields an element in $\map(F(\doublecatS),F(\LinOrd{1} \boxproduct \LinOrd{0})) \cong \map(\doublecatS,\LinOrd{1} \boxproduct \LinOrd{0})$ which is an equivalence on objects. 
    From this we learn that $\doublecatS \cong \LinOrd{1} \boxproduct \LinOrd{0}$ and that this composition actually is an equivalence of double categories. 

    Composing \eqref{morphism-equivalence-2} with \eqref{morphism-equivalence} the other way around now yields an endomorphism of $\doublecat$ which is an equivalence on objects, and therefore also an equivalence.
    We conclude that \eqref{morphism-equivalence} is an equivalence. 
    
    A similar argument shows that $F(\LinOrd{0} \boxtimes \LinOrd{1})$ is either $\LinOrd{0} \boxtimes \LinOrd{1}$ or $\LinOrd{1} \boxtimes \LinOrd{0}$ (and different from $F(\LinOrd{1} \boxtimes \LinOrd{0})$).
    
    A computation shows that any pushout of $\LinOrd{0} \boxproduct \LinOrd{1}$ and $\LinOrd{1} \boxtimes \LinOrd{0}$ along $\LinOrd{0} \boxtimes \LinOrd{0}$ must either be $\LinOrd{1} \boxtimes \LinOrd{1}$ (in case where this is literally the condition of being an adequate factorization system, i.e. the pushout in \eqref{OFSwithRLP} or a similar one obtained by pre-composing with $\horop{(-)}$) or it only contains $3$ objects (because in this case, it turns out that the pushout can be computed in $\PSh(\Delta \times \Delta)$). 
    Because $F$ preserves pushouts as well as the space of objects, we conclude that $F$ preserves $\LinOrd{1} \boxtimes \LinOrd{1}$ and that $F$ must either restrict to the identity or to $\horop{(-)}$ on $\Delta_{\le 1} \times \Delta_{\le 1} \subset \Delta \times \Delta \subset \OFSAd$. 
    
    By the Segal condition, every object in $\Delta \times \Delta$ is a colimit of objects in $\Delta_{\le 1} \times \Delta_{\le 1}$, the value $F(\LinOrd{m} \boxtimes \LinOrd{n})$ hence is determined by the restriction of $F$ to $\Delta_{\le 1} \times \Delta_{\le 1}$. A calculation shows that if $F$ restricts to $\horop{(-)}$ on $\Delta_{\le 1} \times \Delta_{\le 1}$, it will still send $\LinOrd{m} \boxtimes \LinOrd{n}$ to $\LinOrd{m} \boxtimes \LinOrd{n}$ (and also if $F$ restricts to the identity, obviously). 

    We moreover claim that any automorphism of $\Delta \times \Delta$ is already determined by its restriction to $\Delta_{\le 1} \times \Delta_{\le 1}$. Indeed, the $k+1$ inclusions $\LinOrd{0} \to \LinOrd{k}$ and the $m+2$ maps $\LinOrd{m} \to \LinOrd{1}$ induce an injective map \[ \map_{\Delta}(\LinOrd{k},\LinOrd{m}) \longhookrightarrow \left( \map_{\Delta}(\LinOrd{0},\LinOrd{1} \right)^{\times (k+1)(m+2)} \,. \] 

    The category $\DblCatAd$ is a reflective subcategory of $\PSh(\Delta \times \Delta)$, every colimit-preserving functor is hence left Kan extended along the inclusion $\Delta \times \Delta \to \DblCatAd$. Therefore, we have inclusions \[ \Aut(\DblCatAd) \subset \Fun^L(\DblCatAd,\DblCatAd) \subset \Fun(\Delta \times \Delta,\DblCatAd) \,. \]
    It follows from the above discussion that we actually have that \[\Aut(\OFSAd) \subset \Aut(\Delta \times \Delta,\Delta \times \Delta)\] is the subspace spanned by the two components $\horop{(-)}$ and the identity. 
    
    Finally, note that the identity functor on $\Delta \times \Delta$ has no non-trivial automorphisms because none of the objects in $\Delta \times \Delta$ has. 
\end{proof}
\begin{corollary} \label{span=verop}
    Let $\sOFS$ be an adequate factorization system. Then there is a natural equivalence \[ \horop{\DCLR(\sOFS)} \cong \DCLR(\Span(\sOFS)) \] where $\Span \colon \OFSAd \to \OFSAd$ denotes Barwick's span category functor, \cite[Def. 5.7]{bar17}.
\end{corollary}
\begin{proof}
    By \cite[Thm. 4.12]{hhln23a}, $\Span$ is an automorphism of $\OFSAd$ which is not naturally equivalent to the identity and so is $\corners \circ \horop{(-)} \circ \DCLR$ by \Cref{main-comparison-Ad}, so they agree by \Cref{Aut(AOFS)} (using the equivalence from \Cref{main-comparison-Ad} again).
\end{proof}
Combining the above result with \Cref{comparefibrations}, we recover the equivalence between op-Gray fibration over an orthogonal adequate triple and curved orthofibrations between its span category from \cite[Thm. 5.21]{hhln23a}:
\begin{theorem}[Haugseng, Hebestreit, Nuiten, Linskens]
    Let $\sOFS$ be an adequate factorization system. The functor $\Span \colon \OFSAd \to \OFSAd$ induces a natural equivalence 
    \[ \opGray(\sOFS) \cong \Ortho(\Span(\sOFS)) \]
\end{theorem}
\begin{proof}
    This follows from combining \Cref{span=verop} and \Cref{comparefibrations}.
\end{proof}
\sloppy
\printbibliography
\end{document}